\documentclass[graybox]{svmult}

\usepackage{mathptmx}       
\usepackage{helvet}         
\usepackage{courier}        
\usepackage{type1cm}        
\usepackage{makeidx}         
\usepackage{graphicx}        
\usepackage{multicol}        
\usepackage{amsmath, amssymb, bbm}

\usepackage[bottom]{footmisc}

\makeindex             

\newcommand{\eps}{\varepsilon}
\renewcommand{\b}{\beta}
\renewcommand{\d}{\delta}

\newcommand{\e}{\varepsilon}

\renewcommand{\phi}{\varphi}

\newcommand{\s}{\sigma}

\newcommand{\ssup}[1] {{\scriptscriptstyle{({#1}})}}

\newcommand{\Nlt}{\underline{N}(t)}

\newcommand{\Nl}{\underline{N}(t)}

\newcommand{\one}{{\mathbbm 1}}

\newcommand{\Deltad}{\Delta}
\newcommand{\dd}{\,d}

\newcommand{\R}{\mathbb R}
\newcommand{\Z}{\mathbb Z}

\newcommand{\N}{\mathbb N}

\renewcommand{\E}{\mathbb E}
\renewcommand{\P}{\mathbb P}
\newcommand{\heap}[2]  {\genfrac{}{}{0pt}{}{#1}{#2}}
\newcommand{\sfrac}[2] {\mbox{$\frac{#1}{#2}$}}

\begin{document}

\title*{A scaling limit theorem for the parabolic Anderson model with exponential potential}
\titlerunning{A scaling limit theorem for the parabolic Anderson model}
\author{Hubert Lacoin and Peter M\"orters}
\institute{Hubert Lacoin \at Rome \email{lacoin@math.jussieu.fr}
\and Peter M\"orters \at Bath \email{maspm@bath.ac.uk}}
%
%
\maketitle

\abstract*{The parabolic Anderson problem is the Cauchy problem for the heat equation 
$\partial_t u(t,z)=\Delta u(t,z)+\xi(t,z) u(t,z)$ on $(0,\infty)\times {\mathbb Z}^d$ with random potential 
$(\xi(t, z) \colon z\in {\mathbb Z}^d)$ and localized initial condition. In this paper we consider potentials
which are constant in time and independent exponentially distributed in space. 
We study the growth rate of the total mass of the solution in
terms of weak and almost sure limit theorems, and the spatial spread of the mass in terms of a scaling limit theorem.
The latter result shows that in this case, just like in the case of heavy tailed potentials, the mass gets
trapped in a single relevant island with high probability.}

\abstract{The parabolic Anderson problem is the Cauchy problem for the heat equation 
$\partial_t u(t,z)=\Delta u(t,z)+\xi(t,z) u(t,z)$ on $(0,\infty)\times {\mathbb Z}^d$ with random potential 
$(\xi(t,z) \colon z\in {\mathbb Z}^d)$ and localized initial condition. In this paper we consider potentials
which are constant in time and independent exponentially distributed in space. We study the growth rate of the total mass 
of the solution in
terms of weak and almost sure limit theorems, and the spatial spread of the mass in terms of a scaling limit theorem.
The latter result shows that in this case, just like in the case of heavy tailed potentials, the mass gets
trapped in a single relevant island with high probability.}

\ \\[-1.4cm]
\section{Introduction and main results}
\ \\[-2.3cm]
\subsection{Overview and background}
\ \\[-0.7cm]
We consider the heat equation with random potential on the integer lattice $\Z^d$ and study 
the Cauchy problem with localised initial datum,
\begin{align*} 
\begin{array}{rcll}
\displaystyle \vspace{2mm} \partial_t u(t,z) & = & \Deltad u(t,z)+\xi(t,z)\,u(t,z), \qquad
& \mbox{ for }(t,z)\in (0,\infty)\times \Z^d,\\
\displaystyle\lim_{t\downarrow 0} u(t,z) & = & \one_{0}(z), & \mbox{ for } z\in\Z^d,
\end{array}
\end{align*}
where \ \\[-5mm]
\begin{align*}
(\Deltad f)(z)=\sum_{y\sim z} [f(y)-f(z)], \qquad \mbox{ for } z\in\Z^d, f\colon \Z^d\to\R,\\[-8mm]
\end{align*}
is the discrete Laplacian, and the potential $(\xi(t,z)\colon t>0, z\in\Z^d)$ is a random field. 
This equation is known as the \emph{parabolic Anderson model}.\medskip

In the present paper we assume that the potential field is constant in time and independent, identically
distributed in space according to some nondegenerate distribution. Under this hypothesis 
the solutions are believed to exhibit \emph{intermittency}, which roughly speaking
means that at any late time the solution is concentrated in a small number of \emph{relevant islands} 
at large distance from each other, such that the diameter of each island is  much smaller than this distance,
see Figure~1 for a schematic picture. The relevant
islands are located in areas where the potential has favourable properties, e.g. a high density of 
large potential values. As time progresses new relevant islands emerge in locations further and further away from the origin 
at places where the potential is more and more favourable, while old islands lose their relevance. 
The main aim of the extensive research in this model, which was
initiated by G\"artner and Molchanov in~\cite{GM90, GM98}, is to get  a better understanding of the phenomenon of intermittency 
for  various choices of potentials.\medskip
\pagebreak[3]

\begin{figure}[htb]\label{fig1}\begin{center}\scalebox{0.4}{\input{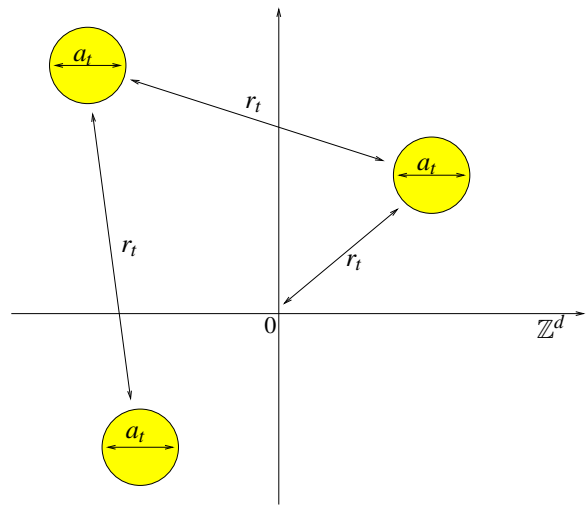}}\end{center} 
\caption{A schematic picture of intermittency: the mass of the solution is concentrated on relevant 
islands (indicated by shaded balls) with radius of order $a_t$ and distances of order $r_t \gg a_t$. }
\end{figure} 

Natural questions about the nature of intermittency are the following:
\begin{itemize}
\item What is the diameter of the relevant islands? Are they growing in time?
\item How much mass is concentrated in a relevant island?\\ How big is the potential on a relevant island? 
\item Where are the relevant islands located? What is the distance of different islands?
\item How many relevant islands are there?
\item How 
do new relevant islands emerge? What is the lifetime of a relevant island?
\end{itemize}

\pagebreak[3]

Explicit answers to these questions and, more generally, results on the precise geometry of solutions to the parabolic 
Anderson model are typically very difficult to obtain. In the related context of Brownian motion among Poissonian
obstacles, Sznitman~\cite{Sz98} provides methodology to study properties of Brownian paths conditioned on survival, which
offer a possible route to the geometry of solutions, at least in the case of bounded potentials. In a seminal paper 
G\"artner, K\"onig and Molchanov~\cite{GKM07} follow a different route to analyse size and position of relevant islands 
in the case of double exponential potentials. Their results also offer some insight into potentials with heavier tails. 
In~\cite{KLMS09} and~\cite{MOS09} a complete picture of the geometry of the solutions is given in the case of Pareto 
distributed potentials, building on the work of~\cite{GKM07}. In this case of an extremely heavy tailed potential it can 
be shown that, for any $\eps>0$  at sufficiently late times, there exists a single point carrying a proportion of mass 
exceeding $1-\eps$ with probability converging to one. This point constitutes the single relevant island and very  precise 
results about the location, lifetime and dynamics of this island can be obtained, see also ~\cite{M10} for a survey
of this research.

For more complicated potentials however, one has to rely on less explicit results. A natural way forward
is to investigate  the growth rates of the \emph{total mass}
$$U(t):= \sum_{z\in\Z^d} u(t,z)$$
of the solution. If the potential is bounded from above we define the (quenched) \emph{Lyapunov exponent} as
$$\lambda := \lim_{t\to\infty} L_t \mbox{ where } L_t:=\frac1t \log U(t),$$
whenever this limit exists in the almost sure sense. If the potential is unbounded one expects superexponential
growth and is interested in an asymptotic expansion of $L_t$. If the tails of the potential distribution are 
sufficiently light so that the logarithmic moment generating function
$$ H(x) := \log E e^{x \xi(0)}$$ 
is finite for all $x\ge 0$, a large deviation heuristics suggests that,we get
$$L_t =  \frac{H( \beta_t\alpha_t^{-d})}{\beta_t\alpha_t^{-d} } - \frac1{\alpha_t^2} \big( \kappa + o(1)\big), 
\mbox{ almost surely as $t\uparrow\infty$,}$$
where $\alpha, \beta$ are deterministic scale functions and $\kappa$ is a deterministic constant.  According to the
heuristics, the quantity $\alpha_t$ can be interpreted as the diameter of the relevant islands at time~$t$, and the 
leading term as the size of the potential values on the island. The constant $\kappa$ is given in terms of a
variational problem whose maximiser describes the shape of a vertically shifted and rescaled potential on an island.
More details and a classification of light-tailed potentials according to this paradigm are given in~\cite{HKM06}.
\pagebreak[3]

If the potential is such that the moment generating functions do not always exist, this approach breaks down. Indeed,
one can no longer expect the leading terms in an expansion of $L_t$ to be deterministic. Instead, one should expect 
the solutions to be concentrated in islands consisting of single sites and the expansion of $L_t$ to reflect fluctuations 
in the size of the potential on these sites. One would expect the sites of the islands to be those with the largest potential in some
time-dependent centred box and the fluctuations to be similar to those seen in the order statistics of independent random 
variables. This programme is carried out in detail in~\cite{HMS08} for potentials with Weibull (stretched exponential) and  
Pareto (polynomial) tails. In the present paper we add the case of standard exponential potentials and present weak 
(see Theorem~\ref{weak}) and almost sure (see Theorem~\ref{surely}) asymptotic expansions for $L_t$ in this case. These
results are taken from the first author's unpublished master thesis~\cite{L07} and were announced without proof in~\cite{HMS08}.
\smallskip

Very little has been done so far to get a precise understanding of the number and position of the relevant islands, 
the very fine results for the Pareto case being the only exception. A natural idea to approach this
with somewhat softer techniques is to prove a scaling limit theorem. To this end we define a probability 
distribution $\nu_t$ on $\Z^d$ associating to each site $z$ a weight proportional to the solution~$u(t,z)$, {i.e.}\\[-2mm]
$$\nu_t := \sum_{z\in \Z^d} \frac{u(t,z)}{U(t)}\, \delta(z), \mbox{ for any } t\ge 0, \\[-2mm]$$
where $\delta(z)$ denotes the Dirac measure concentrated at $z\in \R^d$.
For $a>0$, we also define the distribution of mass at the time~$t$ in the scale~$a$ as 
$$\nu^a_t:= \nu_t\big(\sfrac{\cdot}{a}\big)=  \sum_{z\in \Z^d} \frac{u(t,z)}{U(t)}\,\delta\big(\sfrac{z}{a}\big), \\[-3mm]$$ 
which is considered as an element of the  space ${\mathcal M}(\R^d)$ of probability measures on~$\R^d$. 
Identifying the scale $r_t$ of 
the distances between the islands and the origin, intermittency would imply that islands are contracted to points and
that $\nu_t^{r_t}$ converges in law to a random probability measure, which is purely atomic with atoms representing 
intermittent islands and their weights representing the proportion of mass on the islands. In the case of Pareto 
potentials such a result follows easily from the detailed geometric picture, see \cite[Proposition 1.4]{MOS09}, 
but in principle could be obtained from softer arguments. It therefore seems viable that scaling limit theorems
like the above can be obtained for a large class of potentials including some which are harder to analyse because
they have much lighter tails.
\smallskip

In Theorem~\ref{main} of the present paper we show that in the case of exponential potentials for 
$r_t=t/\log\log t$ the random probability  measures $\nu_t^{r_t}$ converge in distribution to a 
point mass in a nonzero random point.  
In particular this shows that for exponential potential we also have \emph{only one relevant island}. 
Moreover, the  solution of the parabolic Anderson problem spreads \emph{sublinearly} in space. 
Our arguments can be adapted to the easier case of Weibull, or stretched exponential, potentials, where there is 
also only one relevant island but the solution has a \emph{superlinear spread}.
%
These results are new and open up possibilities for further research projects, which we briefly mention in
our concluding remarks.
\pagebreak[3]

\subsection{Statement of results}

We now assume that $(\xi(z) \colon z\in\Z^d)$ is a family of independent random variables with
$$P\big(\xi(z)>x\big)=e^{-x} \mbox{ for $x\ge 0$}.$$
Suppose $(u(t,z) \colon t>0, z\in\Z^d)$ is the unique nonnegative solution to the
parabolic Anderson model with this potential, and let $(U(t)\colon t>0)$ be the total mass
of the solution. We recall that
$$L_t= \frac1t \log U(t)$$
and first ask for a weak expansion of~$L_t$ up to the first nondegenerate random term. This turns out
to be the third term in the expansion, which is of constant order. In the following we use $\Rightarrow$ to 
indicate convergence in distribution.
\smallskip

\begin{theorem}[Weak asymptotics for the growth rate of the total mass]\label{weak}\ \\
We have
$$L_t - d \log t + d \log\log\log t \Rightarrow X,$$
where $X$ has a Gumbel distribution
$$P(X\leq x) = \exp\big\{-2^d e^{-x+2d} \big\} \qquad \mbox{ for } x\in\R.$$
\end{theorem}

In an almost sure expansion already the second term exhibits
fluctuations.

\begin{theorem}[Almost sure asymptotics for the growth rate of the total mass] \label{surely}\ \\
Almost surely,
$$\limsup_{t \uparrow \infty} \frac{L_t - d \log t}{\log\log t}=1,$$
and
$$\liminf_{t \uparrow \infty} \frac{L_t - d \log t}{\log \log\log t}=-(d+1).$$
\end{theorem}

\begin{remark}
Note that neither of these almost sure asymptotics agree with the asymptotics
$$\lim_{t \uparrow \infty} \frac{L_t - d \log t}{\log \log\log t}=-d \quad \mbox{ in probability,}$$ 
which follows from Theorem~\ref{weak}. The almost sure results pick up
fluctuations on both sides of the second term in the weak expansion, with those above being significantly stronger 
than those below the mean. This is different in the stretched exponential case studied in~\cite{HMS08}, where the liminf
behaviour coincides with the weak limit behaviour. The limsup behaviour in the exponential case is included in the 
results of~\cite{HMS08} and therefore not proved here.
\end{remark}

\pagebreak

Recall that the distribution of the mass of the solution at time $t>0$ and on the scale $a>0$ is defined as
a (random) element of the  space ${\mathcal M}(\R^d)$ of probability measures on $\R^d$ by
$$\nu^a_t:= \nu_t\left(\sfrac{\cdot}{a}\right)=  \sum_{z\in \Z^d} \frac{u(t,z)}{U(t)} \, \delta\big(\sfrac{z}{a}\big).$$
The following theorem is the main result of this paper.
\smallskip

\begin{theorem}[Scaling limit theorem]\label{main}
Defining the sublinear scale function
$$r_t=\frac{t}{\log\log t},$$
we have
$$\lim_{t\uparrow \infty} \nu_t^{r_t}= \delta(Y) \mbox{ in distribution,}$$
where $\delta(x)$ denotes the Dirac measure concentrated in $x\in\R^d$ and $Y$ is a random variable 
in $\R^d$ with independent coordinates given by  standard exponential variables with uniform random sign.
\end{theorem}

\begin{remark}
In the case of a Weibull potential with parameter $0<\gamma<1$ given by
$$P\big(\xi(z)>x\big)=e^{-x^\gamma} \mbox{ for $x\ge 0$,}$$
a variant of the proof gives convergence of $\nu_t^{r_t}$ for the superballistic scale function
$$r_t=\frac{t (\log t)^{\frac1{\gamma}-1}}{\log\log t},$$
to a limit measure $\delta(Y)$ where the components of $Y$ are independent
exponentially distributed with parameter~$d^{1-1/\gamma}$ and uniform sign.
Details are left to the reader.
\end{remark}

\section{Proof of the main results}

\subsection{Overview}

The proofs are based on the Feynman-Kac formula 
$$u(t,z)=\E\Big[\exp\Big\{\int_0^t\xi(X_s) \, ds \Big\} \one\{X_t=z\}\Big],$$
where $(X_s \colon s\ge 0)$ is a continuous-time simple random walk on~$\Z^d$ started 
at the origin and the probability $\P$ and expectation~$\E$ refer only to this walk 
and not to the potentials. Recall that $(X_s \colon s\ge 0)$ is the Markov process generated 
by the discrete Laplacian~$\Delta$ featuring in the parabolic Anderson problem. It is shown
in~\cite{GM90} that the Feynman-Kac formula gives the unique solution to the parabolic Anderson problem
under a moment condition on the potential, which is satisfied in the exponential case. 
By summing over all sites the Feynman-Kac  formula implies that the total mass is given by
$$U(t)=\E\Big[\exp\Big\{\int_0^t\xi(X_s) \, ds \Big\} \Big].$$
An analysis of this formula allows us to approximate $L_t=\frac1t \log U(t)$ almost
surely from above and below by variational problems for the potential. These variational
problems have the structure that one optimizes over all sites $z\in\Z^d$ the difference
between the potential value $\xi(z)$, corresponding to the reward for spending time in
the site, and a term corresponding to the cost of getting to the site, which is going
to infinity when $z\to\infty$ and thus ensure that the problem is well-defined. 
\smallskip

We can use the result for the lower bound given in \cite[Lemmas 2.1 and 2.3]{HMS08}. Here
and throughout this paper we use $|\,\cdot\,|$ to denote the $\ell^1$-norm on $\R^d$.

\begin{lemma}[Lower bound on $L_t$]\label{lowbound}
Let
$$\Nlt := \max_{z\in\Z^d} \Big\{ \xi(z)-\frac{|z|}{t}\log \xi(z)\Big\},$$
then, almost surely, for all sufficiently large~$t$, we have
$$L_t\ge \Nlt-2d+o(1).$$
\end{lemma}
 
The appearance of $\xi(z)$ in the cost term can be explained by the fact that part of the cost
arises from the fact that the optimal paths leading to~$z$ spend a positive proportion of the overall time
traveling to the site and therefore miss out on the optimal potential value for some considerable time, 
see Section~1.3 in \cite{HMS08} for a heuristic derivation of this formula.%
\medskip%

\pagebreak[3]

The corresponding upper bound will be our main concern here. 

\begin{lemma}[Upper bound on $L_t$]\label{upbound}
For any $c>0$ let
$$\overline{N}_c(t) := \max_{t/(\log t)^2\le |z|\le t\log t}
\Big\{ \xi(z)-\frac{|z|}{t} \, \big(\log\log |z|+c\big) \Big\}.$$
Then, for any $\eps>0$ there exists $c=c(\eps)>0$ such that, almost surely, 
for all sufficiently large~$t$, 
we have $$L_t\le \overline{N}_{c(\eps)}(t)-2d+\eps+o(1).$$
\end{lemma}


This lemma will be proved in two steps: We first remove paths that do not make an essential
contribution from the average in the Feynman-Kac formula using an ad-hoc approach, see
Lemma~\ref{nohits} and Lemma~\ref{jumps}. Then we 
use the properties of the remaining paths to refine the argument and get an improved bound, see
Proposition~\ref{fineupper}.
\medskip

The variational problems for the upper and lower bound can then be studied using an extreme value 
analysis, which follows along the lines of~\cite{HMS08}. It turns out that the weak and almost sure 
asymptotics of the two problems coincide up to the accuracy required to 
prove Theorem~\ref{weak} and  Theorem~\ref{surely}.
\medskip

For the proof of the scaling limit we need to give an upper bound on the growth rate of the 
contribution of all those paths ending in a site at distance more than $\delta r_t$, for some $\delta>0$, 
from the site with the largest potential among those sites that can be reached by some path with the 
same number of jumps. This bound needs to be strictly better than the lower bound on the overall growth rate. 
To this end, in a first step, we again use Lemma~\ref{nohits} and Lemma~\ref{jumps} to eliminate some paths 
using ad-hoc arguments. In the second step we remove paths that never hit the site with largest potential 
that is within their reach. This is  done on the basis of the gap between the largest and the second largest 
value for the variational problem in the upper bound. In the third step it remains to analyse the contribution 
of paths that  hit the optimal site but then move away by more than $\delta r_t$. Again it turns out that the 
rate of growth of the contribution of these paths is strictly smaller than the lower bound on the growth rate 
of the total mass. Proposition~\ref{fineupper} is set up in such a way that it can deal with both the second
and third step. We conclude from this that the solution is concentrated in a single island of diameter at 
most~$\delta r_t$ around the optimal site. An extreme value analysis characterizes the location of the 
optimal site and concludes the proof of Theorem~\ref{main}.
\medskip

The remainder of the paper is structured as follows: In Section~\ref{se.aux} we give some notation and
collect auxiliary results from \cite{HMS08}. Section~\ref{se.upper} contains the required upper bounds
and constitutes the core of the proof. Section~\ref{se.vari} studies the variational problem arising in the
upper bound. Using these approximations we complete the proof of Theorem~\ref{surely} in Section~\ref{se.surely} 
and of Theorem~\ref{weak} in Section~\ref{se.weak}. The proof of the scaling limit theorem,  Theorem~\ref{main}, 
is completed in Section~\ref{se.main}.\\[-5mm]

\subsection{Auxiliary results}\label{se.aux}

Let $B_r=\{|z|\le r\}$ be the ball of radius $r$ centered at the origin in $\Z^d$.  The number~$l_r$ of points in $B_r$ grows asymptotically 
like $r^d$. More precisely, there exists a constant $\kappa_d$ such that, $\lim_{r\rightarrow\infty}l_rr^{-d}=\kappa_d$. 
We define $M_r=\max_{|z|\le r} \xi(z)$ to be the maximal value of the potential on $B_r$.
The behavior of $M_r$ is described quite accurately in \cite[Lemma~4.1]{HMS08}, which we restate now.

\begin{lemma}[Bounds for $M_r$]
\label{m_extreme}
Let $\delta\in (0,1)$ and $c>0$. Then, almost surely,
\begin{align*}
\begin{array}{rcll}\displaystyle
M_r&\le&d\log r+\log\log r+(\log\log r)^{\delta}
&\text{ for all sufficiently large } r,\\
M_r&\ge&d\log r-(1+c)\log\log\log r&\text{ for all sufficiently large }r.\\
\end{array}
\end{align*}
In particular, for any pair of constant $c_1$ and $c_2$ satisfying $c_1<d<c_2$, we have
$$c_1\log r \leq M_r \leq c_2\log r \quad \mbox{ for all sufficiently large } r.$$
\end{lemma}
\pagebreak[3]

Let $M_r^{\ssup i}$ denote the $i$-th biggest value taken by the potential 
in the ball of radius $r$ centered at the origin. 
The next lemma gives us estimates for upper order statistics for the potential. 


\begin{lemma}[Rough asymptotic behaviour for upper order statistics]\label{estimate}
Let $0<\b<1$ be a fixed constant. Then, almost surely, 
\begin{align*}
\lim_{n\to\infty}\frac{M_n^{\ssup{\lfloor n^{\b}\rfloor}}}{\log{n}}= d-\b.
\end{align*}
\end{lemma}


\begin{proof}
Recalling that $l_n$ is the number of points in a ball of radius $n$ in $\Z^d$ we get
\begin{align}\label{distribution}
P\Big(M_n^{\ssup{\lfloor n^{\b} \rfloor }}\le x \Big)=\sum_{i=0}^{\lfloor n^\b\rfloor-1}\binom{l_n}{i}e^{-xi}\left(1-e^{-x}\right)^{l_n-i} .
\end{align}           
We fix $\e>0$  and infer that
\begin{align*}
P\Big(M_n^{\ssup{\lfloor n^{\b}\rfloor}} & \le \left(d-\b-\e\right) \log n \Big)
\le \sum_{i=0}^{\lfloor n^\b\rfloor}\left(l_nn^{-d+\b+\e}\right)^i\left(1-n^{-d+\b+\e}\right)^{l_n-n^\b}\\
&\le \left(n^\b+1\right)\left((\kappa_d+o(1))n^{\b+\e}\right)^{n^{\b}}\exp\left[-(\kappa_d
+o(1))n^{\b+\e}\right]\\
&=\exp\left[-n^{\b+\e}(\kappa_d+o(1))\right]. 
\end{align*}
Since this sequence is summable, we can use the Borel--Cantelli lemma to obtain the lower bound. 
Similarly, for the upper bound, we use \eqref{distribution} to get
\begin{align}\label{upb}
P\Big(M_n^{(\lfloor n^{\b}\rfloor)}\ge (d-\b+\e)\log n\Big)
&\le \sum_{i=\lfloor n^\b\rfloor}^{l_n}\binom{l_n}{i}n^{-(d-\b+\e)i}.
\end{align}
We now use a rough approximation for the binomial coefficient, namely  
$$\binom{l_n}{i}\le\frac{(l_n)^{i}}{i!}\le \left(\frac{e l_n}{i}\right)^{i},$$
when $i$ is big enough. Combining this with \eqref{upb} and using
that the first term in the ensuing sum is the largest, we obtain, for all sufficiently large~$n$,
\begin{align*}
&P\Big(M_n^{\ssup{\lfloor n^{\b}\rfloor}}\ge (d-\b+\e)\log n \Big)
\le \sum_{i=\lfloor n^\b\rfloor}^{l_n}\left(\frac{ e l_n}{i n^{d-\b+\e}}\right)^{i}
\le l_n \left(\frac{ e l_n}{ n^{d+\e}}\right)^{n^\b} 
\le e^{-n^{\b}}.
\end{align*}
Using the Borel--Cantelli lemma again we obtain an upper bound, 
completing the proof of our statement.
\end{proof}

\bigskip

\pagebreak[3]

Let $0<\s<\rho<\frac{1}{2}$ be some fixed constants. We define
\begin{center}
$k_n=\lfloor n^{\sigma}\rfloor $ \phantom{aaaaaa} and \phantom{aaaaaa} $m_n=\lfloor n^{\rho}\rfloor $
\end{center}
Combining Lemma~\ref{m_extreme} and Lemma~\ref{estimate}, we get the following result.

\begin{lemma} \label{tech}
For any constant $c>0$, for all sufficiently large $n$,  we have
\begin{itemize}
	\item[(i)]\ $  M_n^{\ssup 1}-M_n^{\ssup{k_n}}> (\s-c) \log n$;
	\item[(ii)]\ $ M_n^{\ssup {k_n}}-M_n^{\ssup {m_n}}> (\rho-\s-c)\log n$.
	\end{itemize}
\end{lemma}

Finally, we use Lemma~\ref{m_extreme} to give a lower bound for~$\Nl$.
\begin{lemma}[Eventual lower bound for $\Nlt$] \label{evlb} 
For any small $\e>0$, we have
\begin{align*}
\Nlt \ge d\log t -(d+1+\e)\log\log\log t,
\end{align*}
for all sufficiently large~$t$, almost surely.
\end{lemma}

\begin{proof}
Using Lemma \ref{m_extreme} we get, for any fixed $c>0$ and $c_2>d$,
\begin{align*}
\Nlt \ge \max_{r>0}\left[d\log r-(1+c)\log\log\log r -\frac{r}{t}\log\log r -\frac{r}{t}\log c_2 \right],
\end{align*}
if the maximum of the expression in the square brackets (which we denote by $f_t(r)$) is attained at
a point $r_t$, large enough so that Lemma \ref{m_extreme} holds. 

The solution $r=r_t$ of $f_t'(r)=0$ satisfies
\begin{align*}
\frac{d}{r}=\frac{\log\log r}{t}\, \big(1+{o}(1)\big).
\end{align*}
Writing $r_t=t\phi(r_t)$, where $\phi(r)=d(\log\log r)^{-1}(1+{o}(1))$ we get that 
\begin{align}\label{B}
\log \phi(r)= -\log\log\log r +\log d + o(1)
\end{align}
and hence $\log r_t = \log t + \log \phi(r_t)= \log t + o(\log r_t)$, which implies
$\log r_t/ \log t =1+ o(1)$. Note that this implies $r_t\rightarrow \infty$ as $t\rightarrow \infty$, which  
justifies \textit{a posteriori} the application of Lemma~\ref{m_extreme}. Combining this with \eqref{B} we get,
\begin{align*}
f(r_t)&=d(\log(t\phi(r_t)))-(1+c)\log\log\log r_t -\phi(r_t)(\log\log r_t+\log c_2 )\\
&=d\log t - (1+d+c)\log\log\log t +O(1).
\end{align*}
\end{proof}


\subsection{Upper bounds}\label{se.upper}
\smallskip

We start by showing ad-hoc bounds for the growth rates of the contribution
of certain families of paths. These can be compared to the lower bound for the growth rate
of $U(t)$ showing that the paths can be be neglected.
%
For a path $(X_s \colon s\geq 0)$ on the lattice $\Z^d$ we denote by $J_t$ the number number of
jumps up to time~$t$. 
Recall that $M_n^{\ssup k}$ denotes the $k^{\rm th}$ largest potential value on
the sites $z\in\Z^d$ with $|z|\leq n$.
\smallskip

\begin{lemma}\label{nohits}
Fix $0<\sigma<\frac12$ and $k_n=n^\sigma$. Let
$$U_2(t)=\E\Big[\exp\Big\{\int_0^t\xi(X_s)\dd s\Big\}
\one\Big\{\sfrac{t}{(\log t)^2} \le J_t\le t \log t,
\max_{0\le s\le t}\xi(X_s)\le M_{J_t}^{\ssup{k_{J_t}}}\Big\}\Big].$$
Then
$$\lim_{t\uparrow \infty} \frac1t\, \log \frac{U_2(t)}{U(t)} = -\infty.$$
\end{lemma}

\begin{proof}
Simply replacing $\xi(X_s)$ in the integral by the maximum, we get
 \begin{align*}
U_2(t)&=\sum_{t/(\log t)^{2}\le n\le t\log t} \E\Big[\exp\Big\{\int_0^t\xi(X_s)\dd s\Big\}
\one\big\{J_t=n, \max_{0\le s\le t}\xi(X_s)\le M_{n}^{\ssup{k_n}}\big\}\Big]\\
&\le \sum_{t/(\log t)^{2} \le n\le t\log t}e^{tM_n^{(k_n)}}\, \P(J_t=n)
\le \max_{t/(\log t)^{2} \le n\le t\log t} e^{tM_n^{(k_n)}}.
\end{align*}
By Lemma~\ref{estimate} we have $M_n^{(k_n)}=(d-\s)\log n+o(\log n)$ and hence
$$\frac{1}{t}\log U_2(t)\le (d-\s)\log t + o(\log t),$$
so that the result follows by comparison with Lemma~\ref{lowbound} and Lemma~\ref{evlb}.
\end{proof}

\begin{lemma}\label{jumps}
Let
$$U_3(t)=\E\Big[\exp\Big\{\int_0^t\xi(X_s)\dd s\Big\}
\Big(\one\big\{J_t < \sfrac{t}{(\log t)^2} \big\}+\one\big\{J_t>t \log t \big\}\Big)\Big].$$
Then
$$\lim_{t\uparrow \infty} \frac1t\, \log \frac{U_3(t)}{U(t)} = -\infty.$$
\end{lemma}

\begin{proof}
We first show that almost surely, 
\begin{equation}\label{U3bound}
\frac{1}{t}\log U_3(t)\le \max_{n<t/(\log t)^{2} }\Big\{M_n-\frac{n}{t}\log\frac{n}{2det}\Big\}-2d+o(1).
\end{equation}
Indeed, we have
\begin{align}
U_3(t)&\le \sum_{\heap{\{n<t /(\log t)^2\}}{\cup\{n>t\log t\}}} e^{tM_n}\P(J_t=n)
= \sum_{\heap{\{n<t /(\log t)^2\}}{\cup\{n>t\log t\}}}  e^{tM_n}\frac{(2dt)^{n}e^{-2dt}}{n!}\notag\\
&\le \sum_{\heap{\{n<t /(\log t)^2\}}{\cup\{n>t\log t\}}}\exp \big(tM_n-2dt+n\log2dt-\log n! \big). \label{sum2}
\end{align}
To estimate $n!$ we use Stirling's formula, 
\begin{align*}
n!=\sqrt{2\pi n}\left(\frac{n}{e}\right)^ne^{\delta(n)}, \qquad\mbox{ with }\lim_{n\uparrow\infty}\delta(n)=0.
\end{align*}
Fixing some $\e>0$ we know from Lemma \ref{m_extreme}, that $M_n\le \left(d+\e\right) \log n$ for all sufficiently large~$n$, 
so for $t$ large enough, we obtain for all $n>t \log t$, 
\begin{align*}
tM_n-2dt+n\log 2dt-\log n! &\le t (d+\e)\log n -n\log \sfrac{n}{2edt} -\d (n)\\
&\le t (d+\e)\log n \Big(1-\sfrac{1+o(1)}{(d+\e)}\log \left(\sfrac{\log t}{2ed}\right)+o(1)\Big) \\
&\le -2\log n,
\end{align*}
by noticing that   $n\mapsto\frac{n}{t\log n}\log \frac{n}{2edt}$  is decreasing on  $(t \log t,\infty)$.
Hence, almost surely, $$\sum_{n>t\log t} \exp \left(tM_n-2dt+n\log2dt-\log n! \right)=o(1),$$ 
so that using \eqref{sum2} the following upper bound for $U_3$
\begin{align*}
U_3(t)&\leq \frac{t}{(\log t)^2} \max_ {n<t/(\log t)^2}\exp \left(tM_n-2dt + n\log 2dt - \log n! \right)+o(1)\\
      &\leq \frac{t}{(\log t)^2} \max_ {n<t/(\log t)^2}\exp \left(tM_n-2dt - n\log\sfrac{n}{2edt}+o(t)\right)+o(1)
\end{align*}
and hence \eqref{U3bound} follows. As a second step we show that
\begin{equation}\label{detbound}
\frac{1}{t}\log U_3(t)\le d \log t - (2d-1)\log\log t + o(\log\log t).
\end{equation}
Recall that $r \mapsto M_r$ is a non-decreasing function and check that
\begin{align*}
r\longmapsto \frac{r}{t}\log\frac{r}{2det} \text{ is decreasing on } (0,2det),
\end{align*}
hence, replacing $r$ in the bracket by $t/(\log t)^{2}$
\begin{align*}
\max_{r<t/(\log t)^{2} }\left[M_r-\frac{r}{t}\log\frac{r}{2det}\right]=M_{t/(\log t)^{2}}+o(1).
\end{align*}
By Lemma \ref{m_extreme} we have $M_r \le d\log r + \log\log r + o(\log\log r)$ for all 
sufficiently large~$r$, we get, for $t$ large enough
\begin{align}\label{113}
\max_{r<t/(\log t)^{2} }\left[M_r-\frac{r}{t}\log\frac{r}{2det}\right]\le d\log t -(2d-1)\log \log t +o(\log\log t),
\end{align}
and combining \eqref{U3bound} and \eqref{113}, we have proved~\eqref{detbound}. 
Using Lemma~\ref{lowbound} and Lemma \ref{evlb},
\begin{align*}
\frac{1}{t}\log\frac{U_3(t)}{U(t)} &\le \frac{1}{t}\log U_3(t) - \Nlt -2d + o(1)\\ 
&\le -\left(2d-1\right)\log\log t + o(\log\log t)\rightarrow -\infty,
\end{align*}
and hence our statement is proved.
\end{proof}

The following versatile upper bound is the main tool in the proof of all our theorems and
will be used repeatedly. Note for example that, together with Lemmas~\ref{nohits} and \ref{jumps}
it implies  Lemma~\ref{upbound} if the parameters in~$(ii)$ are chosen as $k=1$ and $\delta=0$.

\begin{proposition}\label{fineupper}
For a path $(X_s \colon s\geq 0)$ on the lattice $\Z^d$ we denote by $J_t$ the number of
jumps up to time~$t$. We denote by $M_n^{\ssup k}$ the $k^{\rm th}$ largest potential value on
the sites $z\in\Z^d$ with $|z|\leq n$, and let $Z^{\ssup k}_n$ be the site where this maximum is attained.
Further fix $0<\sigma<\frac12$ and let $k_n=\lfloor n^\s \rfloor$ and $a_t\downarrow 0$.
\begin{itemize}
\item[(a)]\ 
For $n\in\N$ let
$$\begin{aligned}
U^{\ssup{n}}_1(t)= 
\E\Big[ \exp\Big\{ \int_0^t \xi(X_s) \, ds \Big\} &\one\{ J_t=n\}\,
\one\big\{\max_{0\le s\le t}\xi(X_s) > M_{n}^{\ssup{k_{n}}}\big\} \Big].
\end{aligned}$$
Then, for all $\eps>0$ there exists $C_\eps>0$ such that uniformly for all $t a_t \le n\le t \log t$,
$$\frac 1t \log  U^{\ssup n}_1(t)
\leq M_n^{\ssup 1} - \frac{n}{2t}\,\big( \log\log n - C_\eps\big) + \eps - 2d + o(1)\quad
\mbox{ as $t\uparrow\infty$.}$$
\item[(b)]\ 
For fixed $\delta\geq 0$ and $k,n\in\N$ let
$$\begin{aligned}
U^{\ssup {\delta,k,n}}_1(t)= 
\E\Big[ \exp\Big\{ & \int_0^t \xi(X_s) \, ds \Big\} \\
& \one\{ J_t=n\}\, \one\Big\{ \sup_{0\leq s \leq t} \xi(X_s) \not\in \{M_{n}^{\ssup 1},\ldots, M_{n}^{\ssup{k-1}}\} \Big\}\\
& \one\big\{ Z^{\ssup k}_{n} \in \{X_s \colon 0\leq s \leq t\}, |X_t-Z^{\ssup k}_{n}| \geq \delta r_t\big\} \Big].
\end{aligned}$$
Then,
almost surely, 
$$\mbox{uniformly in $k\le k_n$ and $\frac{t}{(\log t)^2} \leq n \leq t a_t$,}$$
we have that
$$\frac 1t \log  U^{\ssup {\delta,k,n}}_1(t)
\leq  M_n^{\ssup k} - \frac{|Z^{\ssup k}_n|}{t} \,\log\log |Z^{\ssup k}_n|   -2d 
- \delta
+o(1)\quad
\mbox{ as $t\uparrow\infty$.}$$
\end{itemize}
\end{proposition}

The first step in the proof is to integrate out the waiting times of the continuous time
random walk paths. The following fact taken from~\cite{HMS08} helps with this.  
\smallskip

\begin{lemma}\label{uppb}
Let $\eta_0,\dots,\eta_n$ be fixed real numbers attaining their maximum only once,
i.e.\ there is an index $0\le k \le n$ with $\eta_k> \eta_i$  for all $i\neq k$.
Then, for all $t>0$,
\begin{align*}
\int_{\R_+^n}\exp\Big\{\sum_{i=0}^{n-1}t_i\eta_i+\Big(t-\sum_{i=0}^{n-1}t_i\Big)\eta_n\Big\}\one
\Big\{\sum_{i=0}^{n-1} t_i< t \Big\}\, dt_0\dots dt_{n-1}\le e^{t\eta_k}\prod_{i\neq k}\frac{1}{\eta_k-\eta_i}.
\end{align*}
\end{lemma}
\bigskip
  
\pagebreak[3]
  
\begin{proof}
First, we prove the result for the case $k=n$, i.e.\ $\eta_n>\eta_i$ for all $i<n$. We have
\begin{align}
\int_{\R^n}\exp&\Big\{\sum_{i=0}^{n-1}t_i\eta_i+\Big(t-\sum_{i=0}^{n-1}t_i\Big)\eta_n\Big\}
\one\Big\{\sum_{i=0}^{n-1} t_i< t, t_i\ge 0 \forall i\le n-1  \Big\}\, {d}t_0\dots {d}t_{n-1}\notag\\
  &\quad=e^{t\eta_n}\int_{\R_+^n}\exp\Big\{\sum_{i=0}^{n-1}t_i(\eta_i-\eta_n)\Big\}\one
  \Big\{\sum_{i=0}^{n-1} t_i< t \Big\}\, {d}t_0\dots {d}t_{n-1}\notag\\
&\quad\le e^{t\eta_n}\int_{\R_+^n}\exp\Big\{\sum_{i=0}^{n-1}t_i(\eta_i-\eta_n)\Big\}\, {d}t_0\dots{d}t_{n-1}
= e^{t\eta_n}\prod_{i< n}\frac{1}{\eta_n-\eta_i}\notag.
  \end{align}
Now we show that any permutation of the indices does not change the value of the integral 
above and this will be sufficient to prove the statement.
First, it is obvious that transposition of $i$ and $j$ does not change the integral 
if $i , j\le n-1 $. Now we consider the case of a transposition $\tau$ on $j$ and $n$, 
where $j<n$. We change variables such that $t'_i=t_i$ if $i\neq j, i\le n-1$ and 
$t'_j=t-\sum_{i=0}^{n-1}t_i$, and get
\begin{align*}
&\int_{\R_{+}^n}\exp\Big\{\sum_{i=0}^{n-1}t_i\eta_i+\Big(t-\sum_{i=0}^{n-1}t_i\Big)\eta_n\Big\}\one
\Big\{\sum_{i=0}^{n-1} t_i< t \Big\}\, {d}t_0\dots {d}t_{n-1}\\
&=\int_{\R^n}\exp\Big\{\sum_{i=0}^{n-1}t'_i\eta_{\tau(i)}-\Big(t-\sum_{i=0}^{n-1}t'_i\Big)\eta_{\tau(n)}\Big\}\one\\
& \qquad \qquad \times   \Big\{\sum_{i=0}^{n-1} t'_i< t,  t'_i\ge 0 \forall i\le n-1\Big\}\, {d}t'_0\dots\, {d}t'_{n-1},
 \end{align*}  
 which completes the proof.
  \end{proof}

For the proof of Proposition~\ref{fineupper}\,(b) denote by
\begin{align*}
{\mathcal P}^{\ssup {\delta,k,n}} &  = 
\Big\{  y=(y_0,\dots,y_n) \colon y_0=0,\,|y_{i-1}-y_i|=1,\\
& \quad\{ y_0,\dots,y_n\} \cap \{Z_{n}^{\ssup 1},\ldots, Z_{n}^{\ssup{k-1}}\} = \emptyset,
Z^{\ssup k}_{n} \in \{ y_0,\dots,y_n\},  |y_n-Z^{\ssup k}_{n}| \geq \delta r_t\Big\}
\end{align*}
the set of all `good' paths and let $(\tau_i)$ be a sequence of independent, exponentially 
distributed random variables with parameter $2d$. 
\pagebreak[3]

Denote by ${\sf E}$ the expectation with 
respect to $(\tau_i)$. We have
\begin{equation}\begin{aligned}
U^{\ssup {\delta,k,n}}_1(t)=
\sum_{y\in {\mathcal P}^{\ssup {\delta,k,n}}} (2d)^{-n} {\sf E} \Big[\exp\Big\{ & \sum_{i=0}^{n-1}\tau_i\xi(y_{i})+
 \Big(t-\sum_{i=0}^{n-1}\tau_i\Big)\xi(y_n)\Big\} \\
& \times \one\Big\{\sum_{i=0}^{n-1}\tau_i<t,\sum_{i=0}^{n}\tau_i>t\Big\}\Big].
\label{foru1}
\end{aligned}
\end{equation}

In the further proof we apply
Lemma~\ref{uppb} to the values of the potential $\xi$ along
a path $y$. However, to do so we need the maximum of $\xi$ along the path $y$ to be attained only once. Therefore 
we have to modify the potential along the path slightly.

\pagebreak[3]

We fix $y\in {\mathcal P}^{\ssup {\delta,k,n}}$ and let
$$i(y)=\min\big\{i\in\{0,\ldots,n\} \colon  y_i = Z^{\ssup k}_n\big\}$$
be the index of the first instant where the maximum of the potential over
the path is attained. 
Now we define a slight variation of $\xi$ on $y$ in the following way.
Fix $\eps>0$ and define $\xi^{y}\colon\{0,\dots,n\}\to \R$ by $\xi^y_i=\xi(y_i)$ if $i\ne i(y)$, and $\xi^y_{i(y)}=\xi(y_{i(y)})+\e$.
We obtain, using $\xi(y_i)\le \xi^y_i$, that
\begin{align}
{\sf E} &\Big[\exp\Big\{\sum_{i=0}^{n-1}\tau_i\xi(y_{i})+
\Big(t-\sum_{i=0}^{n-1}\tau_i\Big)\xi(y_n)\Big\}\one\Big\{
\sum_{i=0}^{n-1}\tau_i<t,\sum_{i=0}^{n}\tau_i>t\Big\}\Big]\notag\\
&\le {\sf E} \Big[\exp\Big\{\sum_{i=0}^{n-1}\tau_i\xi^y_i+
\Big(t-\sum_{i=0}^{n-1}\tau_i\Big)\xi^y_n\Big\}\one\Big\{
\sum_{i=0}^{n-1}\tau_i<t,\sum_{i=0}^{n}\tau_i>t\Big\}\Big]\notag\\
&=(2d)^{n+1}\int_{\R_+^d}\exp\Big\{\sum_{i=0}^{n-1}t_i\xi^y_i+\Big(t-\sum_{i=0}^{n-1}t_i\Big)\xi^y_n\Big\}\notag\\
& \qquad \qquad \qquad \qquad \one\Big\{\sum_{i=0}^{n-1}t_i<t,\sum_{i=0}^{n}t_i>t\Big\}\, e^{-2d\sum_{i=0}^{n}t_i}\, {d}t_0\dots{d}t_{n-1}{d}t_n\notag\\
&=(2d)^{n}e^{-2dt}\int_{\R_+^d}\exp\Big\{\sum_{i=0}^{n-1}t_i\xi^y_i+\Big(t-\sum_{i=0}^{n-1}t_i\Big)\xi^y_n\Big\}
\one\Big\{\sum_{i=0}^{n-1}t_i<t\Big\}\,{d}t_0\dots{d}t_{n-1}\notag\\
&\le (2d)^{n}e^{-2dt} e^{\xi^y_{i(y)}t}\prod_{i\neq i(y)}\frac{1}{\xi^y_{i(y)}-\xi^y_i},\label{sstuff}
\end{align}
where the last line follows from Lemma~\ref{uppb}. Using the definition of our function $\xi^y$ we get
\begin{align}
e^{\xi^y_{i(y)}t}\prod_{i\neq i(y)} \frac{1}{\xi^y_{i(y)}-\xi^y_i}&=e^{(\xi(y_{i(y)})+\e)t}\prod_{i\neq i(y)}\frac{1}{\e+\xi(y_{i(y)})-\xi(y_i)}
\notag\\ &\le e^{(\xi(y_{i(y)})+\e)t}\e^{-n} \prod_{(\xi(y_{i(y)})-\xi(y_i))>1} \frac{1}{\xi(y_{i(y)})-\xi(y_i)}. \label{ssstuff}
\end{align}
Next recall that 
$\rho$ is fixed, and $m_n=\lfloor n^{\rho}\rfloor$.
Let
$$G_n=\big\{ Z_{n}^{\ssup 1},\ldots, Z_{n}^{\ssup {m_n}} \big\} \subset\{z\in\Z^d \colon |z|\leq n\},$$
and call the complement $G_n^{\rm c}$ the set of sites with very low potential.  Note that there are at 
least  $|Z^{\ssup k}_n|+\lfloor \delta r_t\rfloor-m_n$ points in the path~$y$ that belong to $G_n^{\rm c}$. 
Hence there are at least 
$$ |Z^{\ssup k}_n|+\lfloor \delta r_t\rfloor-m_n$$ 
terms in the product in the left hand side of \eqref{ssstuff} that are smaller than 
$$\big(M^{(k_n)}_{n}-M^{(m_n)}_{n}\big)^{-1}$$ 
provided this is less than 1. Combining this with \eqref{foru1}, \eqref{sstuff} and \eqref{ssstuff}, 
we get
\begin{align*}
U^{\ssup {\delta,k,n}}_1(t) &\le \sum_{y\in {\mathcal P}^{\ssup {\delta,k,n}}} \e^{-n} e^{(M^{(k)}_{n}+\e-2d)t}\big(M^{(k_n)}_{n}-
M^{(m_n)}_{n}\big)^{-|Z^{\ssup k}_n|-\lfloor \delta r_t\rfloor+m_n}\\
&\le (2d)^n \e^{-n}e^{(M^{(k)}_{n}+\e-2d)t}\left(\sfrac{\rho-\s}{2}\log n \right)^
{-|Z^{\ssup k}_n|-\lfloor \delta r_t\rfloor+m_n}.
\end{align*}
Taking the $\log$ of the above and defining $C_\e:=\log(\frac{2d}\e)-\log(\sfrac{\rho-\s}2)$ we get
\begin{align*}
 \frac{1}{t}\log U^{\ssup {\delta,k,n}}_1(t)&\le  \sfrac{n}{t} \log \sfrac{2d}{\e}+M_n^{\ssup k}-2d+\e-\sfrac1t 
 \big(|Z^{\ssup k}_n|+\lfloor \delta r_t\rfloor -m_n\big)  \log \left(\sfrac{\rho-\s}{2}\log n \right)\\
 &\le M_n^{\ssup k}-\sfrac{1}{t}|Z^{\ssup k}_n| \log \log  |Z^{\ssup k}_n|-2d+\e + \sfrac{n}{t} C_{\e}
 -\delta \, \sfrac{\log \log n}{\log\log t} +o(1),
 \end{align*}
where we use that $|Z^{\ssup k}_n|+\lfloor \delta r_t\rfloor\le n$. Observing that
$\log \log n \ge(1+o(1))\log\log t$ and $\frac{n}{t} C_{\e}=o(1)$, uniformly for all $n$ in the given range, 
concludes the proof of~(b).
\bigskip

To prove part~(a) we show that regardless of the distance travelled by the path, it hits a site with very low potential in
every other step. Recall that a set $H$ of vertices of $\Z^{d}$ is $\textit{totally disconnected}$ if there is
no pair of vertices $(x,y)\in H^{2}$ such that $|x-y|=1$.

\begin{lemma} Almost surely, for sufficiently large~$n$, the set $G_{n}$ is totally disconnected.
\end{lemma}

\begin{proof}
We prove the statement for $d\geq 2$ first. 
If $i$ and $j$ are distinct integers in $\{1,\ldots,m_n\}$, the random pair of points $(Z_n^{\ssup i},Z_n^{\ssup j})$ is uniformly 
distributed over all possible pairs of points in the ball of radius $n$. As no vertex has more than $2d$ neighbours, we have
$P(Z_n^{\ssup i}-Z_n^{\ssup j})\le {2d}/{l_n}$.
Summing over all possible pairs $i,j$ we get
\begin{align}
P\big(G_n \ \text{not totally disconnected} \big) &\leq \sum_{i<j}P\big(Z_n^{\ssup i}-Z_n^{\ssup j}\big)
\leq \binom{m_n}{2}\, \frac{2d}{l_n}\leq Cn^{2\rho-d}. \label{Borel}
\end{align}
for some constant $C$. Since $\rho<\frac{1}{2}$ and $d\geq 2$ we can apply the
Borel-Cantelli lemma and obtain the result.
%
We now prove the the same result when $d=1$. We introduce a new quantity 
\begin{center}
 $m'_n=\big\lfloor n^{\rho'}\big\rfloor$ with $\rho<\rho'<\frac{1}{2}$ 
\end{center}
Let $G'_{n}$ be the set of the $m'_n$ vertices in the ball of radius $n$ where the biggest values of $\xi$ are 
taken, and let $p_n$ be the biggest integer power of $2$, which is less than~$n$. 
Note that, by~\eqref{Borel}, the set $G'_{p_n}$ is totally disconnected for all sufficiently large~$n$.

We now prove that
\begin{align}
G_n\subseteq G'_{2 p_n} \quad\text{for all sufficiently large $n$.} \label{inclusion}
\end{align}
For this it suffices to show that at least $m_n$ points of $G'_{2 p_n}$ are in the ball of radius $n$. Indeed, 
if we assume this and also that $G_n\nsubseteq G'_{2 p_n}$ we can find a vertex $z_0$ satisfying, 
$|z_0|\leq n$, $z_0\in G'_{2 p_n}$ and $z_0\notin G_n$.
This implies that every $z\in G_n$ satisfies $\xi(z)>\xi(z_0)$, because $G_n$ is the set where the largest values of $\xi$ are achieved.
Then, because $z_0\in G'_{2 p_n}$, we have 
\begin{align*}
G_n\subseteq \left\{\xi(z)>\xi(z_0)\right\}\cap B_n \subseteq \left\{\xi(z)>\xi(z_0)\right\}\cap B_{2p_n}\subseteq G'_{2 p_n},
\end{align*}
which leads to a contradiction to our assumption.

In fact we will prove the slightly  stronger statement that there are at least $m_{2p_n}$ vertices of 
$G'_{2 p_n}$ in the ball of radius $p_n$, and we will now write $p$ instead of $p_n$.
We write $$G'_{2p}=\big\{a'_0,\dots,a'_{m'_{2p}-1}\big\},$$
where $a'_i$ is the vertex where $\xi(a'_{i})=M_{2p}^{(i+1)}$ and introduce
\begin{center}
$X=(X_i)_{0\leq i\leq m'_{2p}-1}$ with $X_i=\one_{\left\{|a'_i|\leq p\right\}}$  and $|X|=\sum\limits_{i=0}^{m'_{2p}-1} X_i.$
\end{center}
Observing that $m'_{2 p}=o(p)$ and that $G'_{2p}$ is uniformly distributed over all possible ordered sets and recalling that the box of 
radius~$p$ contains~$2p+1$ vertices, it is easy to see that for~$p$ big enough,
\begin{align*}
P\big(X_j=1 \, \big| \, X_i=x_i, \forall i<j \big)<\sfrac{3}{4} \mbox{ and }
P\big(X_j=0 \, \big| \, X=x_i, \forall i<j \big)<\sfrac{3}{4},
\end{align*}
for all $j\leq m'_{2p}-1$ and for all fixed $(x_0,...,x_{j-1})\in \left\{0,1\right\}^{j}$. Hence 
\begin{align*}
P\big( |X|<m_{2p} \big )&=\sum_{i=0}^{m_{2p}-1}\sum_{|x|=i}P(X=x)\\
& \le \sum_{i=0}^{m_{2p}-1} \binom{m'_{2p}}{i} \left(\frac{3}{4}\right)^{m'_{2p}}
\le m_{2p}\left(m'_{2p}\right)^{m_{2p}-1}\left(\frac{3}{4}\right)^{m'_{2p}}\\
 & = \exp\left(-m'_{2p}\log(4/3)+(m_{2p}-1)\log m'_{2p}+\log m_{2p} \right)\\
 &= e^{ -(2p)^{\rho'}(1+o(1))}\le e^{-n^{\rho'}} \quad \text{as } n\leq 2p_n.   
\end{align*}
Using the Borel-Cantelli lemma we can prove \eqref{inclusion}, which implies the statement.
\end{proof}

We define the set of paths $\mathcal P_n$ to be 
\begin{multline}
{\mathcal P}_n   = 
\Big\{  y=(y_0,\dots,y_n) \colon y_0=0,\,|y_{i-1}-y_i|=1,
\\ \{ y_0,\dots,y_n\} \cap \{Z_{n}^{\ssup 1},\ldots, Z_{n}^{\ssup{k_n-1}}\} \ne \emptyset \Big\},
\notag
\end{multline}
so that 
$$\begin{aligned}
U^{\ssup {n}}_1(t)=
\sum_{y\in {\mathcal P}_n} (2d)^{-n} {\sf E} \Big[\exp\Big\{\sum_{i=0}^{n-1}\tau_i\xi(y_{i})+
& \Big(t-\sum_{i=0}^{n-1}\tau_i\Big)\xi(y_n)\Big\} \\
& \times \one\Big\{\sum_{i=0}^{n-1}\tau_i<t,\sum_{i=0}^{n}\tau_i>t\Big\}\Big].
\end{aligned}$$
\pagebreak[3]

We can now argue similarly as for part~(b) but using this time the fact that for any path in $\mathcal P_n$ the number of step out of $G_n$ is at least $\lfloor n/2 \rfloor$.
More precisely,
\begin{align*}
U^{\ssup {n}}_1(t) &\le \sum_{y\in {\mathcal P}_n} \e^{-n} e^{(M^{(1)}_{n}+\e-2d)t}\big(M^{(k_n)}_{n}-
M^{(m_n)}_{n}\big)^{-\lfloor  n/2\rfloor},
\end{align*}
and taking the $\log$ of the above and defining $C_\e:=2\log(\frac{2d}\e)-\log(\sfrac{\rho-\s}2)$ we get
\begin{align*}
 \frac{1}{t}\log U^{\ssup{n}}_1(t)&\le  \sfrac{n}{t} \log \sfrac{2d}{\e}+M_n^{\ssup 1}-2d+\e-\sfrac1t 
 \lfloor  n/2\rfloor  \log \left(\sfrac{\rho-\s}{2}\log n \right)\\
 &= M_n^{\ssup 1}-\sfrac{n}{2t} \big(\log \log n - C_\e\big)-2d+\e +o(1),
 \end{align*}
which concludes the proof of~(a).

\subsection{Analysis of the variational problem}\label{se.vari}

We use the point process framework established in~\cite[Section 2.2]{HMS08}
adapting the approach of~\cite[Chapter 3]{Re87}. We only give an outline of
the framework and sketched proofs here, see~\cite[Section 2.2]{HMS08} for more details. 
\medskip

Observe that $\mu(dy):=e^{-y}\, dy$
is a Radon measure on~$G:=(-\infty,\infty]$. For any $z\in\Z^d$,
$x\in\R$ and $r>0$, we have
$$\begin{aligned}
r^d P\big( \xi(z)- d \log r \geq x\big) & = r^d \, e^{-d \log r - x} = e^{-x} = \mu\big( [x,\infty]\big).
\end{aligned}$$
Define, for any $q, \tau>0$ the set
$H^q_\tau:=\{ (x,y) \in \dot{\R}^d \times G \colon y \geq q|x|+\tau\},$
where $\dot{\R}^d$ is the one-point compactification of $\R^d$.
As in \cite[Lemma 4.3]{HMS08} we see that the point process
$$\zeta_r=\sum_{z\in\Z^d} \delta\big( (\sfrac{z}{r}, \xi(z)- d \log r) \big)$$
converges in law to the Poisson process $\zeta$ with intensity ${\rm Leb}_d \otimes \mu$ in the sense
that, for  any pairwise disjoint compact sets $K_1,\ldots, K_n\subset H^q_\tau$ with ${\rm Leb}_{d+1}(\partial K)=0$, 
we have that $(\zeta_r(K_1), \ldots, \zeta_r(K_n))$ converge in law to  
$$\bigotimes\limits_{i=1}^n{{\rm Poiss}\big({\rm Leb}_d \otimes \mu(K_i)\big)}.$$
We further note that for $z=t^{1+o(1)}$ we have
$$\psi_t(z):= \xi(z)- \frac{|z|}{t} \log\log |z|
= \xi(z) - \frac{|z|}{r_t}\,\big(1+o(1)\big).$$
As in \cite[Lemma 4.4]{HMS08} applied to $T_t(z,x):=(z,x-|z|)$ we infer from 
this the convergence of the point process
$$\varpi_t:=\sum_{z\in\Z^d} \delta\big( (\sfrac z{r_t},\psi_t(z)-d\log r_t) \big)$$
in law to a Poisson process~$\varpi$ with intensity
$$\big({\rm Leb}_d \otimes \mu \big)\circ T_t^{-1} = e^{-|z|-y}\, dz\, dy,$$
where now the compact sets $K_1,\ldots, K_n$ can be chosen from the set
$H_\tau:=\dot{\R}^{d+1} \setminus (\R^d \times (-\infty, \tau)).$ The form
of these and the previous domains, and in particular the use of the compactification,
ensure that we can use these convergence results to analyse the right hand side of the final
formula in Proposition~\ref{fineupper}.

\begin{lemma}\label{distance}
Let $X_t^{\ssup1}$ and $X_t^{\ssup 2}$ be the sites corresponding to
the largest and second largest value of $\psi_t(z)$, $z\in\Z^d$.
Then $\psi_t(X_t^{\ssup1}) - \psi_t(X_t^{\ssup2})$ converges in law
to a standard exponential random variable.
\end{lemma}

\begin{proof}
Using careful arguments in the convergence step we obtain, for any $a\geq 0$,
$$\begin{aligned}
P\big( \psi_t(X_t^{\ssup1}) & - \psi_t(X_t^{\ssup2}) \geq a \big)\\
& = \sum_y P\big( \varpi_t\big(\R^d \times (y,\infty)\big) =0, \varpi_t(\R^d \times \{y\})=1, \varpi_t\big(\R^d \times (y-a,y)\big)=0 \big)\\
& \rightarrow 
\int  P\big( \varpi\big(\R^d \times (y,\infty)\big) =0\big) P\big(\varpi\big(\R^d \times (y-a,y)\big)=0 \big) e^{-y}\, dy\\
& = \int \exp(-e^{-y+a}) e^{-y}\, dy = e^{-a}.
\end{aligned}$$
\end{proof}

\begin{lemma}\label{point}
Let $X_t^{\ssup1}$ be the site corresponding to
the largest value of $\psi_t(z)$, $z\in\Z^d$.
Then $X_t^{\ssup1}/r_t$ converges in law to a random variable 
in $\R^d$ with coordinates given by independent standard exponential variables
with uniform random signs.
\end{lemma}

\begin{proof}
As above we obtain, for any $A\subset \R^d$ Borel
with ${\rm Leb}_d(\partial A)=0$,
$$\begin{aligned}
P\Big( \frac{X_t^{\ssup1}}{r_t} \in A \Big)
& = \sum_y P\big( \varpi_t\big(\R^d \times (y,\infty)\big) =0, \varpi_t\big(A \times \{y\}\big)=1\big)\\
& \rightarrow 
\int_A dz \int dy \,  e^{-|z|-y}\,  P\big( \varpi\big(\R^d \times (y,\infty)\big) =0\big)\\
& = \int_A  dz \int dy \, \exp(-e^{-y}) e^{-y-|z|} = \int_A 2^{-d}\,e^{-|z|}\, dz.
\end{aligned}$$
Observe that this implies that the limit variable has the given distribution.
\end{proof}
\pagebreak[3]

\subsection{Proof of the almost sure asymptotics}\label{se.surely}

Note that combining Lemma~\ref{lowbound} and Lemma~\ref{evlb} 
establishes the almost sure lower bound for the liminf result in Theorem~\ref{surely}.
To find a matching upper bound, recall from Lemma~\ref{upbound} that, for sufficiently large~$t$,
$$L_t \leq N_\eps(t) - 2d + \eps$$
for ${N}_\eps(t):=\overline{N}_{c(\eps)}(t)$. 
We now approximate the distribution of $N_\e (t)$. 

\begin{lemma}[Approximation for the distribution of $N_\e (t)$]\label{approx}
Let $b_t\uparrow \infty$, then 
\begin{align*}
\log\left(P\big(N_\e (t)\le b_t \big)\right)=-e^{-b_t }r_t^{d}2^{d}\left(1+o(1)\right).
\end{align*}
\end{lemma}

\begin{proof}
Observe that
\begin{align*}
P\left(N_\e (t)\le b_t \right)&= \prod_{t/(\log t)^{2}\le |z| \le t \log t}  \!\!\!\!\! F\left(b_t +\frac{|z|}{t}\left(\log\log |z|-C_{\e}\right)\right).
\end{align*}
The values which $|z|$ can take are such that
$\log\log |z|=\log\log t + o(1)$ 
uniformly for all $z$, and since $b_t \rightarrow\infty$, we have,
\begin{align*}
\log\big(P & \big(N_\e (t)\le b_t \big)\big)\\
&=\!\!\!\!\!\sum_{t/(\log t)^{2}\le |z| \le t \log t}\!\!\!\!\!  \log\Big(1-\exp\big(-b_t -\sfrac{|z|}{t}\big(\log\log t-C_{\e}+o(1)\big)\big)\Big)\notag\\
&=-\left(1+o(1)\right)\!\!\!\!\!\!\sum_{t/(\log t)^{2}\le |z| \le t \log t}\!\!\!\!\!e^{-b_t -\frac{|z|}{r_t}\left(1+o(1)\right)}\notag\\
&=-e^{-b_t }r_t^{d}\left(1+o(1)\right)\int_{\R^d}e^{-{|x|}(1+o(1))}\one_{\left\{\log\log t/(\log t)^{2}\le |x| \le  \log t\log\log t\right\}}\, \text{d}x
\end{align*}
To obtain our final result, we apply the dominated convergence theorem to the integral, which converges to $2^{d}$.
\end{proof}

We are now ready to prove the upper bound. We consider a sequence of times $t_n:=\exp(n^{2})$ for 
which  $N_{\e}(t_n)$ are independent random variables, in order to use Borel-Cantelli. 

\begin{lemma}[Upper bound for lower envelope of $N_\e(t_n)$]\label{Nestio} 
For any small $c>0$, almost surely there are infinitely many $n$ such that
\begin{align*}
N_{\e}(t_n)\le d\log t_n - (1+d-c) \log\log\log t_n.
\end{align*}
\end{lemma}

\begin{proof}
Note that $(N_{\e}(t_n))_{n\ge N}$ is a sequence of independent variables if $N$ is large enough. To see this
it suffices to notice that the different $(N_{\e}(t_n))_{n\ge N}$ depend on the values of the potential
on disjoints areas. Indeed 
\begin{align*}
\frac{t_{n+1}}{\left(\log t_{n+1}\right)^{2}}&=\frac{\exp\left(n^2+2n+1\right)}{(n+1)^4}> n^2\exp\left(n^2\right)=t_n \log t_n \mbox{ for all large $n$.}
\end{align*}
Now we use Lemma \ref{approx} with $b_t= d\log t - (1+d-c) \log\log\log t$ and we get,
\begin{align*}
\log\big(P\big(N_\e(t_n)\le b_{t_n}\big)\big)&=-2^{d}\left(\log\log t_n\right)^{1-c}\left(1+o(1)\right)
\ge -\log n,
\end{align*}
for all sufficiently large $n$. Hence the sum over the probabilities diverges and we obtain
our result by applying the converse of the Borel-Cantelli lemma.
\end{proof}


\subsection{Proof of the weak asymptotics}\label{se.weak}

To prove Theorem \ref{weak} we show that the upper and lower bounds we found earlier for $L_t$ 
both satisfy the required limit statement. We first state the result of~\cite[Proposition 4.12]{HMS08},
which describes the limit result for the lower bound~$\Nlt$.

\begin{lemma}[Weak asymptotics for $\Nlt$]\label{waN} As $t$ tends to infinity,
\begin{align*}
\Nlt-d\log t + d\log\log\log t \Rightarrow X, \quad\text{where}\quad P(X\le x)=\exp\big(-2^{d}e^{-x}\big).
\end{align*}
\end{lemma}

Next we check the analogous limit theorem for the upper bound~$N_\e(t)$ and thus complete 
the proof of Theorem~\ref{weak}.
  
\begin{lemma}[Weak asymptotics for $N_\e(t)$]\label{wa} As $t$ tends to infinity,
\begin{align*}
N_\e(t)-d\log t + d \log\log\log t \Rightarrow X,\quad\text{where}\quad P(X\le x)=\exp\big(-2^{d}e^{-x}\big).
\end{align*}
\end{lemma}

\begin{proof}
Fix $x\in\R$  and apply Lemma~\ref{approx} with $b_t=d\log t - d \log\log\log t +x$ to get
\begin{align*}
\log\big(P\big(N_\e(t)-d\log t + d \log\log\log t\le x \big)\big)=-e^{-x}2^{d}\,\left(1+o(1)\right),
\end{align*}
which proves our result.
\end{proof}


\subsection{Proof of the scaling limit theorem}\label{se.main}

We recall that $X_t^{\ssup k}$ $(k=1,\ 2)$ is the site at which 
$$\psi_t(z)= \xi(z)- \frac{|z|}{t} \log\log |z|$$
takes its $k^{\rm th}$ largest value. Fix $\delta>0$ and write
$$U(t)=U_1(t)+U_2(t)+U_3(t)+U_4(t)+U_5(t)+U_6(t),$$
where $U_2$ and $U_3$ were defined in Lemma~\ref{nohits}, resp.\ Lemma~\ref{jumps}, and 
\begin{align*}
U_1(t) &=\E\Big[ \exp\Big\{ \int_0^t \xi(X_s) \, ds \Big\} 
\one\big\{\sfrac{2t}{(\log t)^2} \le J_t\le t a_t,
\max_{0\le s\le t}\xi(X_s) > M_{J_t}^{\ssup{k_{J_t}}}\big\}\\
&\phantom{ = \E\Big[ \exp\Big\{ \int_0^t \xi(X_s) \, ds \Big\} }\hspace{2mm} 
\one \big\{ X^{\ssup 1}_t \in \{X_s \colon 0\leq s \leq t\}, |X_t-X^{\ssup 1}_t| \leq \delta r_t\big\}\Big],\\
U_4(t) &= \E\Big[ \exp\Big\{ \int_0^t \xi(X_s) \, ds \Big\} 
\one\big\{ t a_t < J_t \le t \log t,  \max_{0\le s\le t}\xi(X_s) > M_{J_t}^{\ssup{k_{J_t}}}\big\}\Big],\\
U_5(t) &= \E\Big[ \exp\Big\{ \int_0^t \xi(X_s) \, ds \Big\} 
\one\big\{\sfrac{2t}{(\log t)^2} \le J_t\le t a_t,
\max_{0\le s\le t}\xi(X_s) > M_{J_t}^{\ssup{k_{J_t}}}\big\}\\
&\phantom{ = \E\Big[ \exp\Big\{ \int_0^t \xi(X_s) \, ds \Big\} }\hspace{2mm} 
\one\big\{ X^{\ssup 1}_t \not\in \{X_s \colon 0\leq s \leq t\} \big\} \Big],\\
U_6(t) &= \E\Big[ \exp\Big\{ \int_0^t \xi(X_s) \, ds \Big\} 
\one\big\{\sfrac{2t}{(\log t)^2} \le J_t\le t a_t,
\max_{0\le s\le t}\xi(X_s) > M_{J_t}^{\ssup{k_{J_t}}}\big\}\\
&\phantom{ = \E\Big[ \exp\Big\{ \int_0^t \xi(X_s) \, ds \Big\} }\hspace{2mm} 
\one\big\{ X^{\ssup 1}_t \in \{X_s \colon 0\leq s \leq t\}, |X_t-X^{\ssup 1}_t| > \delta r_t\big\}\Big].
\end{align*}
Observe that our result follows if the contributions of $U_i(t)$ for $i=2, \ldots, 6$ to the total mass are negligible,
as $U_1(t)$ only contributes to the mass distributed on points close to $X_t^{\ssup 1}$ on the $r_t$ scale. 

\begin{lemma}\label{4and5}
Suppose $a_t \downarrow 0$ and $a_t \log\log t \to \infty$. Then we
have, in probability,
$$\lim_{t\uparrow \infty} \frac{U_4(t)}{U(t)} =
\lim_{t\uparrow \infty} \frac{U_5(t)}{U(t)} = 
\lim_{t\uparrow \infty} \frac{U_6(t)}{U(t)} = 0.$$
\end{lemma}

\begin{proof}
For the first statement we use Proposition~\ref{fineupper}\,(a) to see that
$$\frac1t\, \log U_4(t) \le
\sup_{n \ge ta_t} \Big\{M_n^{\ssup 1} - \frac{n}{2t}\,\big( \log\log n - C_\eps\big)\Big\} + \eps - 2d+o(1).$$
By Lemmas~\ref{point} and~\ref{distance} the limit of the right hand side is strictly smaller
than the growth rate of $U(t)$, proving that the first limit in the statement equals zero.
%

Using Proposition~\ref{fineupper}\,(b) with $\delta=0$ and summing over all 
$1\le k \le t^\sigma$ with  $X_t^{\ssup 1} \not= Z_n^{\ssup k},$
and over all $n$ with $2t/(\log t)^2\le n\le t a_t$ we get
$$\frac 1t \log  U_5(t) \leq \max_{z \setminus\{X_t^{(1)}\}} \psi_t(z)  -2d +o(1) = \psi_t(X_t^{\ssup2})-2d +o(1) 
\quad \mbox{in probability. }$$
By Lemma~\ref{distance} we find $\epsilon>0$ such that, with a probability arbitrarily
close to one 
$$\frac 1t \log  U_5(t) \leq \psi_t(X_t^{\ssup1})-2d - \epsilon +o(1),$$
and a comparison with the lower bound $\Nlt$ for the growth rate of $U(t)$ proves the second result.
\pagebreak[3]
 
For the third statement we use  Proposition~\ref{fineupper}\,(b) with the choice of
$\delta>0$ from the statement. Summing over all $1\le k \le t^\sigma$ and $n$ with
$2t/(\log t)^2\le n\le t a_t$ we get, as above,
$$\frac 1t \log  U_6(t) \leq \psi_t(X_t^{\ssup1}) -2d - \delta+o(1).$$ 
We can now argue as before that this rate is strictly smaller
than the lower bound $\Nlt$ for~$U(t)$, proving the final statement.
\end{proof}
\smallskip

We can now complete the proof of Theorem~\ref{main}. By definition we have
\begin{align*}
1 & \geq \liminf_{t\uparrow\infty} \nu_t\big\{z\in \Z^d \ \big| \ |z-X^{\ssup1}_t| \leq \delta r_t\big\}
 \geq \liminf_{t\uparrow\infty} \frac{U_1(t)}{U(t)}=
 1- \limsup_{t\uparrow\infty} \sum_{j=2}^6 \frac{U_j(t)}{U(t)}.
\end{align*}
Combining Lemmas~\ref{nohits},~\ref{jumps} and~\ref{4and5} we see that the limsup is zero, so that
we get
$$\lim_{t\uparrow\infty}\nu_t\big\{z\in \Z^d \ \big| \ |z-X^{\ssup1}_t| \leq \delta r_t\big\}=1
\quad \mbox{ in probability.}\\[-3mm]$$
Combining this with the convergence of $X_t^{\ssup 1}/r_t$ given in
Lemma~\ref{point} and recalling that $\delta>0$ was arbitrary concludes the proof.

\section{Concluding remarks}

It would be interesting  to study scaling limit theorems for potentials with lighter tails 
and thus  shed further light on the number of relevant islands in these cases. 
\smallskip

The techniques of the present paper appear suitable to treat cases where the relevant islands are single sites, 
which is the case for potentials heavier than the double-exponential distributions. For the double-exponential distribution
itself and lighter tails, arguments related to classical order statistics of i.i.d.\ random variables need
to be replaced by eigenvalue order statistics for the random Schr\"odinger operator $\Delta + \xi$ 
on $\ell^2(\Z^d)$, making the problem much more complex. Work in an advanced state of progress by Biskup 
and K\"onig~\cite{BK10} deals with the double-exponential case and strongly hints at localization in a single
island of finite size in this and other cases of unbounded potentials. 
\smallskip

For \emph{bounded} potentials the question of the number of relevant islands and the formulation of a scaling limit 
theorem at present seems wide open  and constitutes an attractive research project. Sznitman in~\cite{Sz95} discusses
an `elliptic version' of the Anderson problem, describing Brownian paths in a Poissonian potential conditioned 
to reach a remote location. Sznitman's technique of enlargement of obstacles, described in~\cite{Sz98}, offers a 
possible approach to the scaling limit theorem, leading in~\cite{Sz97} to a study of fluctuations of the
principal eigenvalues of the operator $\Delta + \xi$ and moreover an analysis of variational problems
somewhat similar to those that we expect to arise in the proof of a scaling limit theorem. 
\smallskip

In the light of our result and this discussion it would be of particular interest to know whether there at
all exist potentials which lead to more than one relevant island, and  if so, to 
find the nature and location of the transition between phases of one and several islands.
\medskip

{\bf Acknowledgements:} Special thanks are due to the organizers of the 
\emph{Workshop on Random Media}, in celebration of J\"urgen G\"artner's 60th birthday, which provided
an ideal forum for discussing the problems raised in this paper. The first author acknowledges the support of ERC grant PTRELSS. The second author is grateful for 
the support of EPSRC through an Advanced Research Fellowship.

\vspace{0.8cm}

\end{document}